\documentclass[a4paper,12pt]{article}
\usepackage{fullpage}
\usepackage{amsfonts}
\usepackage{amssymb,amsthm,amsmath,color}
\usepackage{array}
\usepackage{mathrsfs}
\usepackage{dsfont}
\usepackage{lmodern}
\usepackage{hyperref}
\usepackage{graphicx}
\usepackage[utf8]{inputenc}  
\usepackage[T1]{fontenc} 
\usepackage{stmaryrd}  
\numberwithin{equation}{section}

\DeclareMathAlphabet{\mathbbo}{U}{bbold}{m}{n}

\newtheorem{theorem}{Theorem}[section]

\newtheorem{Remarque}[theorem]{Remark}

\newtheorem{prop}[theorem]{Proposition}

\newtheorem*{thm}{Theorem}

\date{}
\title{Sidon set in a union of intervals}

\author{RIBLET Robin}
\begin{document}
\maketitle
\begin{abstract}
We study the maximum size of Sidon sets in unions of integers intervals. If $A\subset\mathbb{N}$ is the union of two intervals and if $\left| A \right|=n$ (where $\left| A \right|$ denotes the cardinality of $A$), we prove that $A$ contains a Sidon set of size at least $0,876\sqrt{n}$. On the other hand, by using the small differences technique, we establish a bound of the maximum size of Sidon sets in the union of $k$ intervals.
\end{abstract}
 
\section{Introduction}

A Sidon set of integers is a subset of $\mathbb{N}$ with the property that all sums of two elements are distinct. Working on Fourier series, Simon Sidon \cite{Simon_Sidon} was the first to take an interest in these sets. He sought to bound the size of the largest Sidon set in $\left\llbracket1,n\right\rrbracket$. The question has been intensively studied and today it is well known (see \cite{HalberstamRoth}) that the maximum size of a Sidon set in an interval of size $n$ is asymptotically equivalent to $\sqrt{n}$. We denote by $F\big(\left\llbracket 1,n\right\rrbracket\big)$ this maximum size. The lower bound was obtained independently by Chowla \cite{Chowla} and Erd\H{o}s \cite{ErdosTuran} who etablished 
$$\liminf\limits_{\substack{n\rightarrow +\infty}}\dfrac{F\big(\left\llbracket 1,n\right\rrbracket\big)}{\sqrt{n}}\geqslant 1.$$ 
For the upper bound, Erd\H{o}s and Tur\'an \cite{ErdosTuran} proved that $F\big(\left\llbracket 1,n\right\rrbracket\big)<\sqrt{n}+O\left( n^{1/4}\right)$. This was sharpened by Lindström \cite{Lindstrom_ameliore_E-T} who proved that $F\big(\left\llbracket 1,n\right\rrbracket\big)<n^{1/2}+n^{1/4}+1$.
Finally, very recently, Balogh, Füredi and Roy \cite{BaloghRoyFurediArXiv} obtained 
$$F\big(\left\llbracket 1,n\right\rrbracket\big)<\sqrt{n}+0,998n^{1/4} .$$
In this paper we are interested in the size of the largest Sidon set contained in the union of two intervals. For $A\subseteq\mathbb{N}$ we denote by $F(A)$ the maximal cardinality of a Sidon set in $A$. Erd\H{o}s conjectured that $F(A)\geqslant \sqrt{n}$ for all sets $A$ of size $n$. For more readability, if $f$ and $g$ are two fonctions such that $f(n)\geqslant (1-o(1))g(n)$, we will write $f(n)\gtrsim g(n)$. In the same way if $f(n)\leqslant (1+o(1))g(n)$, we will write $f(n)\lesssim g(n)$.
Abbott \cite{Abbott} proved that $F(A)\gtrsim 0,0805\sqrt{n}$ and so, if $I_1$ and $I_2$ are two intervals of respective cardinalities $n_1$ and $n_2$ 
$$F( I_1\cup I_2 ) \gtrsim 0,0805\sqrt{n_1+n_2}.$$
We will prove (Theorem \ref{S_dans2intervalles}) that
$$F( I_1\cup I_2 ) \gtrsim 0,876\sqrt{n_1+n_2}.$$
Conversely, we will show that $F( I_1\cup I_2 )\lesssim\sqrt{n_1+n_2}$ and more generally we will give a bound for the maximum cardinality of a Sidon set in a union of $k$ intervals (Theorem \ref{S_dans_intervalles}). In our result the number $k$ of intervals can grow up with the size of $A$.

\section{Lower bound for the size of the largest Sidon set contained in the union of two intervals} 

Previous works by Singer \cite{Singer}, Chowla \cite{Chowla}, Erd\H{o}s and Tur\'an \cite{ErdosTuran}, lead to 
\begin{equation}\label{Sdans(1,n)}
F\left(\left\llbracket 1,n\right\rrbracket\right)\sim \sqrt{n}.
\end{equation}
Since Sidon's property is stable by translation, if $A$ is an interval of size $n$, \eqref{Sdans(1,n)} proves that $F(A)\sim \sqrt{n}$. We shall study the case where $A$ is the union of two intervals. We could simply choose a Sidon set in the largest of the two intervals.
Therefore if $A=I_1\cup I_2$ where $I_1$ and $I_2$ are disjoint intervals of size $n_1$ and $n_2$ such that $n_1\geqslant n_2$, by \eqref{Sdans(1,n)} we get a Sidon set of size $\sqrt{n_1}$ in $I_1$, which yields
$$F(A)\geqslant\sqrt{n_1}=\frac{1}{\sqrt{2}}\sqrt{2n_1}\geqslant\frac{1}{\sqrt{2}}\sqrt{n_1+n_2}> 0,707\sqrt{n_1+n_2}.$$
We shall get a more precise result in the following statement.
\begin{theorem}\label{S_dans2intervalles}
Let $I_1$ and $I_2$ be two disjoint intervals of respective cardinalities $n_1$ and $n_2$. We have
$$F( I_1\cup I_2 ) \gtrsim 0,876\sqrt{n_1+n_2}.$$
\end{theorem}
\begin{proof}
Let $A$ be the union of two disjoint intervals of respective cardinalities $n_1$ and $n_2$. Since Sidon's property is stable by translation and symmetry, even if it means translating and considering $ A '= \max A-A+1 $, we can assume that $ A = I_1 \sqcup I_2 $ where $I_1=\left\llbracket 1,n_1 \right\rrbracket$, and $ I_2 $ is an interval of cardinality $n_2\leqslant n_1$. \\
$$\begin{minipage}[l]{17cm}
\includegraphics[height=1.4cm]{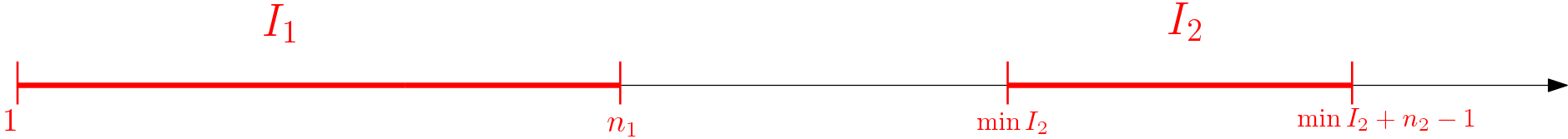}
\end{minipage}$$
$$ $$
\underline{Strategy} : We are going to discuss according to two parameters : the size of $ I_2 $ compared to $ I_1 $, and the distance between $ I_1 $ and $ I_2 $. For that we will consider 
$$\alpha=\frac{n_2}{n_1} \ \text{ and } \ \beta=\frac{\min I_2-n_1}{n_1}.$$
We will distinguish several cases. First of all, if $ \alpha $ is less than a certain level $ \alpha_0 $ (which we will have to optimize at the end of the proof) then we will only have to choose a large Sidon set in $ I_1 $. Indeed, if $ \alpha $ is small, then $ I_2 $ is small in front of $ I_1 $. We will therefore not need its contribution to choose our Sidon set. If, on the other hand, $ \alpha $ is greater than $ \alpha_0 $, then we will distinguish two more cases depending on the size of $ \beta $. If $ \beta $ is less than a certain level $ \beta_0 $ (which we will also have to optimize at the end of the proof), then $ I_2 $ is sufficiently close to $ I_1 $. To get a big Sidon set in $ I_1 \cup I_2 $, we will remove the middle elements : those included in $ \left\llbracket n_1+1, \min I_2-1 \right\rrbracket $ to a big Sidon set in $ \left\llbracket 1, \max I_2 \right\rrbracket $. We will use Singer's famous theorem \cite{Singer} (see also \cite{HalberstamRoth} chapter II) to find a large Sidon set in $ \left\llbracket 1, \max I_2 \right\rrbracket $ with few elements in $ \left\llbracket n_1+1, \min I_2-1 \right\rrbracket$. Finally if $ \beta $ is greater than $ \beta_0 $, we will transform a Sidon set in $ \left\llbracket 1, n_1 + n_2 \right\rrbracket $ to obtain a large Sidon set in $ I_1 \cup I_2 $.

\vspace{0.5cm}

Let $\alpha_0\in \left]0,1\right]$, $\beta_0\in\mathbb{R}_+$, $\alpha=\frac{n_2}{n_1}$ et $\beta=\frac{\min I_2-n_1}{n_1}$. \\

\vspace{0.3cm}

$$\begin{minipage}[l]{17cm}
\includegraphics[height=2cm]{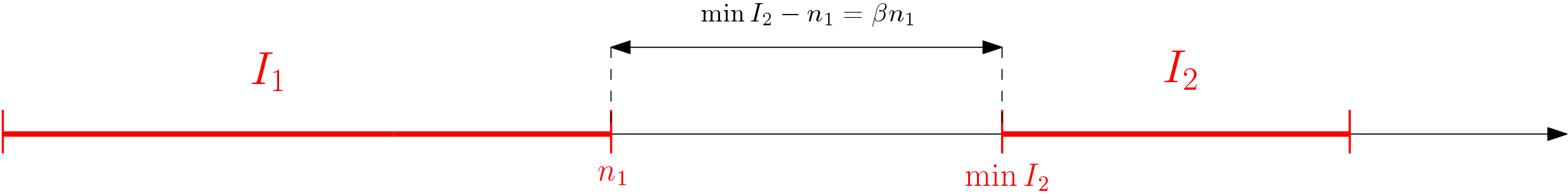}
\end{minipage}$$
$$ $$
\textbf{i)\underline{ If $\alpha\leqslant\alpha_0$}.} 

\vspace{0.5cm}

Write $n_2=\alpha n_1$. It suffices then to choose a Sidon set $ S $ in $ I_1 $ of size $ \sqrt{n_1} $. In this way, we have
$$ F(A) \gtrsim \sqrt{n_1} \gtrsim \frac{1}{\sqrt{1+\alpha}}\sqrt{n_1+n_2} \gtrsim \frac{1}{\sqrt{1+\alpha_0}}\sqrt{n_1+n_2}  .$$
Finally since $\left| A \right|=n_1+n_2$, in this case we get
\begin{equation}\label{alpha<alpha_0}
F(A)\gtrsim \frac{1}{\sqrt{1+\alpha_0}}\sqrt{\left| A \right|}.
\end{equation}

\vspace{0.5cm}

\textbf{ii)\underline{ If $\alpha\geqslant\alpha_0$ and $\beta\leqslant\beta_0$}.} 

\vspace{0.5cm}

We write $n_2=\alpha n_1$ again and we recall that $\beta=\frac{\min I_2-n_1}{n_1}$. In this way, if $n=\max A$, we have
\begin{equation}\label{n=1+alpha+beta)n1}
n= (1+\alpha+\beta)n_1.
\end{equation}

\vspace{0.3cm}

$$\begin{minipage}[l]{17cm}
\includegraphics[height=2.5cm]{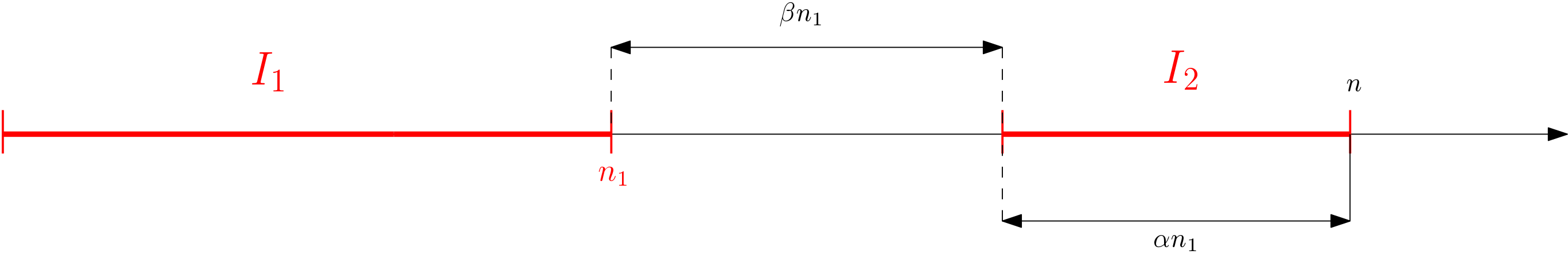}
\end{minipage}$$
$$ $$
As explained before, we want to use Singer's theorem.
\begin{thm}[Singer, \cite{Singer}]\label{thm_Singer}
Let $p$ be a prime. Then there exist $p+1$ Sidon sets $S_1,...,S_p$ each of size $p+1$ such that
$$\bigcup\limits_{\substack{i=1}}^p S_i=\left\lbrace 1,...,p^2+p+1 \right\rbrace.$$
\end{thm}
We want to use it in $\left\llbracket 1,n \right\rrbracket$, so we need to approach $ n $ by $p^2+p+1$ where $ p $ is a prime number. Let $ p $ and $ p '$ be the two consecutive prime numbers such that
$$p^2+p+1\leqslant n <p'^2+p'+1.$$
Since $ p $ and $ p '$ are consecutive, it is well known for instance that $ p'-p = O(p^{5/8}) $ (see \cite{p-p'=O(p5/8)}). (Actually, better results exist on the distance between two consecutive primes (see \cite{p-p'=O(best)})
but this bound is enough for us).
We have $p^2+p+1\leqslant n$. According to Singer's theorem (theorem \ref{thm_Singer}), there exist $p+1$ Sidon set $S_i $ ($i=1,...,p+1$) each of size $p+1$, whose union is $\left\llbracket 1,p^2+p+1 \right\rrbracket$. 
Since $n=p^2+O(p'^2-p^2+p'-p)=p^2+O(p^{13/8})$, we have
$$\min I_2=(1+\beta)n_1=\frac{1+\beta}{1+\alpha+\beta}n=\frac{1+\beta}{1+\alpha+\beta}p^2+O(p^{13/8})\leqslant p^2+p+1,$$
for sufficiently large $n$. Therefore $\left\rrbracket n_1,\min I_2 \right\llbracket\subset\left\llbracket 1,p^2+p+1 \right\rrbracket$, thus 
$$\bigcup\limits_{\substack{i=1}}^{p+1}\big( S_i\cap\left\rrbracket n_1,\min I_2 \right\llbracket\big)=\left\rrbracket n_1,\min I_2 \right\llbracket,$$ and
$$\sum\limits_{\substack{i=1}}^{p+1}\big| S_i\cap\left\rrbracket n_1,\min I_2 \right\llbracket\big|=\beta n_1+o(n_1) .$$
So there exists $i\in\left\llbracket 1,p+1\right\rrbracket$ such that $S=S_i$ satisfies
$$\big| S\cap\left\rrbracket n_1,\min I_2 \right\llbracket\big|\leqslant \dfrac{\beta}{p+1}n_1+o\left(\frac{n_1}{p}\right).$$
Finally, with $S'=S\setminus \left\rrbracket n_1,\min I_2 \right\llbracket$, we have $S'\subset A$ and
$$\left| S' \right|\geqslant p+1-\dfrac{\beta}{p+1}n_1+o\left(\frac{n_1}{p}\right).$$
Now $p\sim\sqrt{n}$, and so by \eqref{n=1+alpha+beta)n1} we get
\begin{align*}
F(A) & \geqslant \left| S' \right| \\ & \geqslant\frac{1+\alpha}{1+\alpha+\beta}\sqrt{n}+o(\sqrt{n})  \\ & \geqslant\sqrt{\frac{1+\alpha}{1+\alpha+\beta}}\sqrt{n_1+n_2}+o(\sqrt{n_1+n_2})\\ & \geqslant\sqrt{\frac{1+\alpha}{1+\alpha+\beta_0}}\sqrt{n_1+n_2}+o(\sqrt{n_1+n_2}).
\end{align*}
Moreover the function which associates $\sqrt{\dfrac{1+x}{1+x+\beta_0}}$ to $ x $ is increasing and we are in case $\alpha\geqslant \alpha_0$, so finally
\begin{equation}\label{alpha>_beta<}
F(A)\gtrsim\sqrt{\frac{1+\alpha_0}{1+\alpha_0+\beta_0}}\sqrt{\left| A \right|}.
\end{equation}

\vspace{0.5cm}

\textbf{iii)\underline{ If $\alpha\geqslant\alpha_0$ and $\beta\geqslant\beta_0$}.} 

\vspace{0.5cm}

Here, we will distinguish between the cases $\beta_0\geqslant 1$ and $\beta_0<1$.

\vspace{0.5cm}

\textbf{iii.a. If $\beta_0\geqslant 1$.} 

\vspace{0.5cm}

Then $ \beta \geqslant 1 $ and therefore $ I_1 $ and $ I_2 $ are sufficiently far apart. \\
$$\begin{minipage}[l]{17cm}
\includegraphics[height=1.5cm]{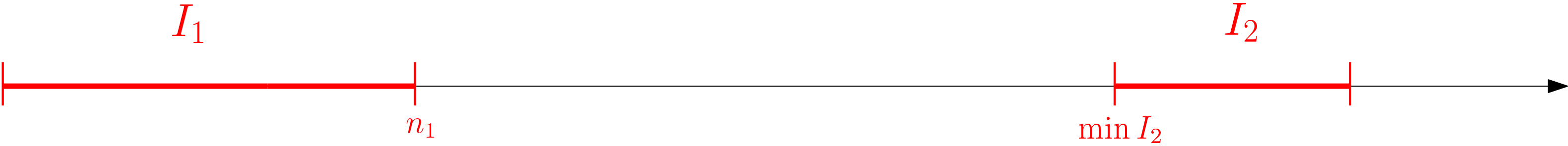}
\end{minipage}$$
$$ $$
We then choose a set of Sidon $ S $ in $\left\llbracket 1,n_1+n_2 \right\rrbracket$ and we define the new set $ S '$ by
$$S'=S_1\sqcup S_2,$$
where $S_1=S\cap\left\llbracket 1,n_1 \right\rrbracket$ and $S_2=S\cap\left\rrbracket n_1,n_1+n_2 \right\rrbracket+\min I_2-n_1$. So $S'\subseteq A$ and we will see that $ S' $ is a Sidon set. First note that $ S_1 $ and $ S_2 $ are Sidon sets, $\max S_1\leqslant n_1$ and $\min S_2>\min I_2\geqslant 2n_1$. Therefore for $a,b\in S'$, $a\neq b$, we have
\begin{equation}\label{disjonction1}
a,b\in S_1\Leftrightarrow a+b<2n_1,
\end{equation}
\begin{equation}\label{disjonction2}
a,b\in S_2\Leftrightarrow a+b>4n_1.
\end{equation}
Let $ a, b, c, d \in S '$ be such that $ a + b = c + d $. We will distinguish between the following three cases : $ a $ and $ b $ both belong to $ S_1 $, both to $ S_2 $, or one belongs to $ S_1 $ and the other to $ S_2 $.
\begin{itemize}
\item If $a,b\in S_1$, then by \eqref{disjonction1} $c,d\in S_1$ and since $ S_1 $ is a Sidon set, $\left\lbrace a,b \right\rbrace=\left\lbrace c,d \right\rbrace$. 
\item If $a,b\in S_2$, then by \eqref{disjonction2} $c,d\in S_2$ and $ S_2 $ is a Sidon set, so $\left\lbrace a,b \right\rbrace=\left\lbrace c,d \right\rbrace$. 
\item If $a\in S_1$ and $b\in S_2$, then as seen in previous arguments, necessarily $ c $ and $ d $ cannot belong both to $ S_1 $ nor both to $ S_2 $. Suppose therefore without lost of generality that $ c \in S_1 $ and $ d \in S_2 $. So we have
$$a+b=c+d \Leftrightarrow a+(b-\min I_2+n_1)=c+(d-\min I_2+n_1),$$
and $a,(b-\min I_2+n_1),c,(d-\min I_2+n_1)\in S$ by construction of $S_1$ and $S_2$. So $\left\lbrace a,(b-\min I_2+n_1) \right\rbrace=\left\lbrace c,d-\min I_2+n_1) \right\rbrace$ because $S$ is a Sidon set. Moreover, since $a,c\in S_1$ and $b,d\in S_2$, we have $a,c\in \left\llbracket 1,n_1 \right\rrbracket$ and $(b-\min I_2+n_1),(d-\min I_2+n_1)\in \left\rrbracket n_1,n_1+n_2 \right\rrbracket$. Hence $a=c$ and $(b-\min I_2+n_1)=(d-\min I_2+n_1)$ so finally $a=c$ and $b=d$. 
\end{itemize}
In conclusion, in any case, we get $\left\lbrace a,b \right\rbrace=\left\lbrace c,d \right\rbrace$, which means that $ S ' $ is a Sidon set. It suffices then to notice that $  \left| S' \right|= \left| S \right| $ and to recall that in $ \left\llbracket 1, n_1 + n_2 \right\rrbracket $, we have Sidon sets of size $ \sqrt{ n_1 + n_2} $, to be able to conclude that when $ \min I_2-n_1 \geqslant n_1 $, we have
\begin{equation}\label{alpha>_beta>_a}
F(A)\gtrsim \sqrt{\left| A \right|}.
\end{equation}

\vspace{0.5cm}

\textbf{iii.b. If $\beta_0< 1$ and $\beta>2\alpha-1$.} 

\vspace{0.5cm}

Let $S$ be a Sidon set in $\left\llbracket 1,\left\lfloor\frac{1+\beta}{2}n_1\right\rfloor+n_2 \right\rrbracket$, and define $S_1=S\cap\left\llbracket 1,\left\lfloor\frac{1+\beta}{2}n_1\right\rfloor \right\rrbracket$, 
$$ S_2=\left( S\cap\left\rrbracket \left\lfloor\frac{1+\beta}{2}n_1\right\rfloor,\left\lfloor\frac{1+\beta}{2}n_1\right\rfloor+n_2 \right\rrbracket\right) +\left\lceil\frac{1+\beta}{2}n_1\right\rceil$$ 
and $S'=S_1\sqcup S_2$. $S_1\subseteq I_1$ and $S_2\subseteq I_2$ so $ S ' \subseteq A $ and we will see that $ S ' $ is a Sidon set. First note that $ S_1 $ and $ S_2 $ are Sidon sets. Moreover, we have $\max S_1\leqslant\frac{1+\beta}{2}n_1$, \\
$\min S_2\geqslant \min I_2+1=(1+\beta)n_1+1$, and since $\beta >2\alpha-1$,
$$\max S_1+\max S_2\leqslant \left(\frac{3}{2}(1+\beta)+\alpha\right)n_1<\left( 2+2\beta\right) n_1.$$ 
So we get as in \textbf{iii.a}, for $a,b\in S'$
$$a,b\in S_1 \Leftrightarrow a+b\leqslant (1+\beta)n_1,$$
$$a,b\in S_2 \Leftrightarrow a+b> (2+2\beta)n_1.$$
Therefore if $\beta>2\alpha-1$, as the previous case, we prove that $S'$ is a Sidon set. So if $\beta>2\alpha-1$, $F(A)\geqslant \left| S' \right| =\left| S \right|$ and we know that we can choose $S$ such that
$$
\left| S \right|  \gtrsim \sqrt{\frac{1+\beta}{2}n_1+n_2}  \gtrsim \sqrt{\frac{1+\beta}{2(1+\alpha)}n+\frac{\alpha}{1+\alpha}n} \gtrsim \sqrt{\dfrac{1+2\alpha+\beta}{2(1+\alpha)}}\sqrt{n},
$$
so finally, if $\beta>2\alpha-1$, we have
\begin{equation}\label{alpha>_beta>_b}
F(A)\gtrsim \sqrt{\dfrac{1+2\alpha_0+\beta_0}{2(1+\alpha_0)}}\sqrt{\left| A \right|}.
\end{equation} 

\vspace{0.5cm}

\textbf{iii.c. If $\beta_0< 1$ and $\beta\leqslant 2\alpha-1$.} 

\vspace{0.5cm}

This time we choose a Sidon set $S$ in $\left\llbracket 1,\left\lfloor\frac{2}{3}(1+\alpha+\beta)n_1\right\rfloor \right\rrbracket$, and define ${S'=S_1\sqcup S_2}$ where $S_1=S\cap\left\llbracket 1,\left\lfloor\frac{1+\alpha+\beta}{3}n_1\right\rfloor \right\rrbracket$, and 
$$S_2=\left( S\cap\left\rrbracket \left\lfloor\frac{1+\alpha+\beta}{3}n_1\right\rfloor,\left\lfloor\frac{2}{3}(1+\alpha+\beta)n_1\right\rfloor \right\rrbracket\right) +\left\lceil\frac{1+\alpha+\beta}{3}n_1\right\rceil .$$ 
Since $\alpha\leqslant 1$ and in the current case $\beta\leqslant 1$, we have $1+\alpha+\beta\leqslant 3$ and so $S_1\subseteq I_1$. Moreover, $\left\lfloor\frac{1+\alpha+\beta}{3}n_1\right\rfloor+1+\left\lceil\frac{1+\alpha+\beta}{3}n_1\right\rceil\geqslant \frac{2}{3}(1+\alpha+\beta)$ and here $\beta\leqslant 2\alpha-1$ so $\frac{2}{3}(1+\alpha+\beta)\geqslant 1+\beta$ and so $S_2\subseteq I_2$. Therefore $ S ' \subseteq A $ and like in the two previous cases, we prove that $ S ' $ is a Sidon set. Finally, in this case, we get the bound
\begin{equation}\label{alpha>_beta>_b_2}
F(A)\gtrsim \sqrt{2\dfrac{1+\alpha_0+\beta_0}{3(1+\alpha_0)}}\sqrt{\left| A \right|}.
\end{equation} 

\newpage

\textbf{iv)\underline{ Conclusion}.} 

\vspace{0.5cm}

Whatever the case we are in, by \eqref{alpha<alpha_0}, \eqref{alpha>_beta<}, \eqref{alpha>_beta>_a}, \eqref{alpha>_beta>_b} and \eqref{alpha>_beta>_b_2}, we have
$$F(A)\gtrsim \min\big( m_1(\alpha_0,\beta_0),m_2(\alpha_0,\beta_0) \big)\sqrt{\left| A \right|} ,$$
where
$$m_1(\alpha_0,\beta_0)=\min\limits_{\substack{\beta_0>2\alpha_0-1}}\max \left( \frac{1}{\sqrt{1+\alpha_0}},\sqrt{\frac{1+\alpha_0}{1+\alpha_0+\beta_0}},\sqrt{\dfrac{1+2\alpha_0+\beta_0}{2(1+\alpha_0)}} \right),$$
and
$$m_2(\alpha_0,\beta_0)=\min\limits_{\substack{\beta_0\leqslant2\alpha_0-1}}\max \left( \frac{1}{\sqrt{1+\alpha_0}},\sqrt{\frac{1+\alpha_0}{1+\alpha_0+\beta_0}},\sqrt{2\dfrac{1+\alpha_0+\beta_0}{3(1+\alpha_0)}} \right) .$$
Optimizing the choices of $ \alpha_0 $ and $ \beta_0 $, we get
$$m_1(\alpha_0,\beta_0)\geqslant \sqrt{\dfrac{\sqrt{13}+1}{6}}\geqslant 0,876 ,$$
reached by $(\alpha_0,\beta_0)=\left( \frac{\sqrt{13}-3}{2},4-\sqrt{13} \right)$, and
$$m_2(\alpha_0,\beta_0)\geqslant \left(\frac{2}{3}\right)^{1/4}\geqslant 0,903 ,$$
reached by $(\alpha_0,\beta_0)=\left( \frac{1+\sqrt{6}}{5},\frac{2\sqrt{6}-3}{5} \right)$.
Finally $m_1(\alpha_0,\beta_0)<m_2(\alpha_0,\beta_0)$ which ends the proof.
\end{proof}

\section{Upper bound for the maximum size of a Sidon set in a union of intervals}

In the previous section we gave a lower bound for the maximum size of a Sidon set in a union of two intervals. Conversely, we seek in this section an upper bound for the maximum size of a Sidon set in a union of intervals. If we consider two intervals of size $n/2$ for example, in each of these two intervals, we can only choose at most (asymptotically) $\sqrt{n/2}$ elements because otherwise it would contradict \eqref{Sdans(1,n)}. A trivial asymptotic bound would therefore be $2\sqrt{n/2}=\sqrt{2}\sqrt{n}$. Using the Erd\H{o}s-Tur\'an small difference technique  \cite{HalberstamRoth}, we can go down to $\sqrt{n}$. 
Actually, we can prove the result for a fixed number of intervals and even for an increasing number of intervals if it remains $o\left( \sqrt{n}\right) $. This is the content of the following theorem.
\begin{theorem}\label{S_dans_intervalles}
If $E$ is a set of cardinality $n\in\mathbb{N}^*$ and $E$ is a union of $k$ intervals, then any Sidon included in $E$ has size at most \\
i) $\left(\alpha+\sqrt{2+\alpha^2} \right)\sqrt{n} +o(\sqrt{n})$ if $\limsup\limits_{\substack{n\rightarrow +\infty}}\dfrac{k}{\sqrt{n}}=\alpha > 0$ \\
ii) $\sqrt{n}+o(\sqrt{n})$ si $k=o(\sqrt{n})$ \\
iii) $\sqrt{n}+\sqrt{k}n^{1/4}+o(n^{1/4})$ if $k=o(n^{1/4})$.
\end{theorem}
\begin{proof}
Let $n,k\in\mathbb{N}^*$ be such that $k\leqslant n$, and
$$E=\bigsqcup\limits_{\substack{i=1}}^k \left\llbracket n^-_i,n^+_i-1 \right\rrbracket ,$$
where $n^-_1<n^+_1<n^-_2<n^+_2<...<n^-_k<n^+_k$ and $\sum\limits_{\substack{i=1}}^k \big( n^+_i-n^-_i\big) =n$. Let $S\subseteq E$ be a Sidon set. For $u$ an integer such that $u<n$, we define the set $\mathcal{M}$ by
$$\mathcal{M}=E+\left\llbracket 1,u \right\rrbracket=\left(\bigsqcup\limits_{\substack{i=1}}^k\left\llbracket n^-_i,n^+_i-1 \right\rrbracket\right)+\left\llbracket 1,u \right\rrbracket.$$ 
We have $\left|\mathcal{M}\right|\leqslant\sum\limits_{\substack{i=1}}^k \big( u+n^+_i-n^-_i\big)=n+ku$.
For $m\in\mathcal{M}$, we consider the intervals $I_m$ defined by 
$$I_m=\left\llbracket m-u,m-1 \right\rrbracket.$$
Let $r=\left| S \right|$. Since each element of $S$ occurs in exactly $u$ intervals of type $I_m$, we have
\begin{equation}\label{II.1.1.1}
\sum\limits_{\substack{m\in \mathcal{M}}}\left|I_m\cap S\right|=ru.
\end{equation}
Thus by the Cauchy-Schwarz inequality, we obtain
\begin{align*}
(ru)^2 & \leqslant\left(\sum\limits_{\substack{m\in \mathcal{M}}} 1\right)\left(\sum\limits_{\substack{m\in \mathcal{M}}}\left|I_m\cap S\right|^2\right) \\ & \leqslant (n+ku)\left(\sum\limits_{\substack{m\in \mathcal{M}}}\left|I_m\cap S\right|^2\right),
\end{align*}
and so
\begin{equation}\label{II.1.1.2}
\sum\limits_{\substack{m\in \mathcal{M}}}\left|I_m\cap S\right|^2\geqslant\dfrac{(ru)^2}{n+ku}.
\end{equation}
For $u<n$ and $m\in \mathcal{M}$, we define
$$T_u(m)=\left|\left\lbrace (s_1,s_2) \ \vert \ s_1,s_2\in \big(S\cap I_m\big) \ , \ s_1<s_2 \right\rbrace\right| ,$$ 
and
$$T_u=\sum\limits_{\substack{m\in \mathcal{M}}}T_u(m).$$
On the one hand, by \eqref{II.1.1.1} and \eqref{II.1.1.2}, we have
$$ T_u = \sum\limits_{\substack{m\in \mathcal{M}}}T_u(m) = \sum\limits_{\substack{m\in \mathcal{M}}}\binom{\left|I_m\cap S\right|}{2} \geqslant \dfrac{1}{2}\left(\dfrac{(ru)^2}{n+ku}-ru\right),$$
which yields
\begin{equation}\label{min_T_u}
T_u\geqslant \dfrac{ru}{2}\left(\dfrac{ru}{n+ku}-1\right).
\end{equation}
On the other hand, for any couple $(s_1,s_2)$ counted in $T_u$, $s_2-s_1$ is an integer $d$ satisfying $0<d<u$. Moreover, since $S$ is a Sidon set, for each $d$, there is at most one matching $(s_1, s_2)$. Finally, a pair $(s_1, s_2)$ corresponding to a certain $d$ appears in exactly $u-d$ intervals $I_m$. Thereby
$$ T_u \leqslant \sum\limits_{\substack{d=1}}^{u-1}(u-d) =\dfrac{u(u-1)}{2} .$$
Using \eqref{min_T_u}, we get
$$\dfrac{ru}{2}\left(\dfrac{ru}{n+ku}-1\right) \leqslant \dfrac{u(u-1)}{2} ,$$
which leads to
$$  r^2u-(n+ku)r\leqslant (u-1)(n+ku),$$
and finally
\begin{equation}\label{avant_choix_u}
r\leqslant\sqrt{\frac{u-1}{u}(n+ku)+\dfrac{(n+ku)^2}{4u^2}}+\dfrac{n+ku}{2u}.
\end{equation}
We just have to choose different values for $u$ according to the relative size of $k$ compared to $n$ in order to conclude.
\begin{itemize}
\item If $\limsup\limits_{\substack{k\rightarrow +\infty}}\dfrac{k}{\sqrt{n}}=\alpha\neq 0$, then we choose $u=\left\lceil\sqrt{n}/\alpha\right\rceil$ and \eqref{avant_choix_u} gives
\begin{align*}
r & \leqslant \dfrac{\alpha\sqrt{n}}{2}+\frac{k}{2}+\sqrt{n+\frac{k\sqrt{n}}{\alpha}+\left(\frac{\alpha\sqrt{n}}{2}+\frac{k}{2} \right)^2+o(n)} \\ & \leqslant \sqrt{n}\left(\alpha+\sqrt{2+\alpha^2} \right) +o(\sqrt{n}).
\end{align*}
\begin{Remarque}\label{S_dans_intervalles_alpha=1/sqrt(2)}
As we want $ r $ to be an $ O (\sqrt{n}) $, in \eqref{avant_choix_u}, on the one hand the $ k (u-1) $ in the root forces us to choose $ u = O (\sqrt{n}) $, and on the other hand the $ \frac{n}{2u} $ outside the root leads us to choose $ \sqrt{n} = O (u) $. So necessarily, our choice will be of the form $u=\gamma \sqrt{n}$.
The choice $u=\left\lceil\sqrt{n}/\alpha\right\rceil$ is the simplest giving a good bound for all $ \alpha $, but at this stage, if we know precisely $ \alpha $, it is possible to do a little better. For exemple if $\alpha=1/\sqrt{2}$, then choosing $u=\left\lceil\beta\sqrt{n}\right\rceil$ and injecting in \eqref{avant_choix_u}, we obtain a function to be minimized in $\beta$. For $\beta=1,79$, we get
$$r\leqslant 2,266\sqrt{n},$$
whereas our general choice $u=\left\lceil\sqrt{n}/\alpha\right\rceil$, only yields
$$r\leqslant 2,29\sqrt{n}.$$
Similarly if $\alpha=1$, we are led to choose $\beta=\frac{1+\sqrt{5}}{2}$ chich gives
$$r\leqslant \frac{3+\sqrt{5}}{2}\sqrt{n}\leqslant 2,62\sqrt{n}.$$
\end{Remarque}
\item If $k=o(\sqrt{n})$, then we choose $u=\left\lceil \dfrac{n^{3/4}}{\sqrt{k}}\right\rceil$ and \eqref{avant_choix_u} gives
\begin{align}\label{cas2versCas3}
r & \leqslant  \sqrt{k}\dfrac{n^{1/4}}{2}+\dfrac{k}{2}+\sqrt{n+\sqrt{k}n^{3/4}+\left(\sqrt{k}\dfrac{n^{1/4}}{2}+\dfrac{k}{2} \right)^2}  \\ & \leqslant \sqrt{n}\sqrt{1+\frac{\sqrt{k}}{n^{1/4}}+o(1)}+o(\sqrt{n}) \notag \\ & \leqslant \sqrt{n}+o(\sqrt{n}). \notag
\end{align}
If $k=o(n^{1/4})$, we can make the error term more precise. \\
\\
\item If $k=o(n^{1/4})$, \eqref{cas2versCas3} gives
\begin{align*}
r & \leqslant \sqrt{k}\dfrac{n^{1/4}}{2}+\dfrac{k}{2}+\sqrt{n+\sqrt{k}n^{3/4}+\left(\sqrt{k}\dfrac{n^{1/4}}{2}+\dfrac{k}{2} \right)^2} \\ & \leqslant \sqrt{k}\dfrac{n^{1/4}}{2}+\dfrac{k}{2}+\sqrt{n}\sqrt{1+\frac{\sqrt{k}}{n^{1/4}}+\frac{k}{4\sqrt{n}}+\frac{k^{3/2}}{2n^{3/4}}+\frac{k^2}{4n}} \\ & \leqslant \sqrt{n}+\sqrt{k}n^{1/4}+o(n^{1/4}),
\end{align*}
where the last line comes from the Taylor expansion $\sqrt{1+x}=1+\frac{x}{2}+o(x)$. 
\end{itemize}
\end{proof}

\section{Conclusion and Remarks}

Theorems \ref{S_dans2intervalles} and \ref{S_dans_intervalles} prove that if $A$ is the union of two intervals of respective size $n_1$ and $n_2$, the maximum cardinality of a Sidon set in $A$ is (asymptotically) between $0,8444\sqrt{n_1+n_2}$ and $\sqrt{n_1+n_2}$. Erd\H{o}s' conjecture claims that it should be equivalent to $\sqrt{n_1+n_2}$. Therefore, it should be very interesting to improve Theorem \ref{S_dans2intervalles} in order to try to bring the constant $0,8444$ closer to $1$. 

It is also surely possible to improve the first point of Theorem \ref{S_dans_intervalles} but we will never be able to reach $\sqrt{n}$. Indeed, it is easy to build Sidon sets with a cardinality larger than $\sqrt{n}$ under the hypothesis of Theorem \ref{S_dans_intervalles}.
\begin{prop}\label{S_dans_n_intervalles}
Let $n\in\mathbb{N}^*$. There exists a Sidon set of size $2n$ in a union of $n$ intervals each of size $n$.
\end{prop}
\begin{proof}
Let $n\in\mathbb{N}^*$, $S_1=\left\lbrace 2^{k+1} \ \vert \ k=1,...,n \right\rbrace$, $S_2=\left\lbrace 2^{k+1}+k \ \vert \ k=1,...,n-1 \right\rbrace$ and $S=S_1\sqcup S_2$. Since $\left| S \right|=2n-1$ and
$$S\subseteq \bigsqcup\limits_{\substack{k=1}}^{n}\left\llbracket 2^{k+1},2^{k+1}+n-1 \right\rrbracket,$$
we just have to check that $S$ is a Sidon set. Let $a,b,c,d\in S$ be such that $a+b=c+d$.

Since $2^{k+2}-2^{k+1}>2k$, $a+b=c+d$ implies that there exists $k_1$ and $k_2$ in $\left\lbrace 1,...,n \right\rbrace$ such that $\left\lbrace a,b \right\rbrace$ and $\left\lbrace c,d \right\rbrace$ are in $\left\llbracket 2^{k_1+1},2^{k_1+1}+n \right\rrbracket\cup\left\llbracket 2^{k_2+1},2^{k_2+1}+n \right\rrbracket$. We can assume without lost of generality that $a,c\in \left\llbracket 2^{k_1+1},2^{k_1+1}+n \right\rrbracket$ and $b,d\in\left\llbracket 2^{k_2+1},2^{k_2+1}+n \right\rrbracket$. In this way, we have
\begin{align*}
a+b=c+d & \Rightarrow a-c=d-b \\ & \Rightarrow a-c \in\left\lbrace 0,k_2 \right\rbrace,
\end{align*}
But since $a,c\in \left\llbracket 2^{k_1+1},2^{k_1+1}+n \right\rrbracket$, $a-c \in\left\lbrace 0,k_1 \right\rbrace$. Therefore either $a=c$ and so $\left\lbrace a,b \right\rbrace=\left\lbrace c,d \right\rbrace$, or $k_1=k_2$ and so $\left\lbrace a,b \right\rbrace=\left\lbrace c,d \right\rbrace$, which ends the proof.
\end{proof}
This proposition proves that the condition $ k = o(\sqrt{n}) $ of the point $ ii) $ of Theorem \ref{S_dans_intervalles} is optimal in a way. To improve this theorem, it would be necessary to reduce the constant $\left(\alpha+\sqrt{2+\alpha^2} \right)$ in front of $ \sqrt{n} $ in the bound of point $ i) $. However, Proposition \ref{S_dans_n_intervalles} implies that we cannot go below $ \sqrt{2} $ for $ \alpha = 1 $.

\newpage
\bibliographystyle{plain}
\bibliography{/Users/robin/Desktop/mabiblio.bib}
\end{document}